\documentclass[a4paper,11pt]{amsart}

\usepackage{graphics}

\usepackage[T1]{fontenc}    

\usepackage{amsthm}
\usepackage{amsbsy,amsmath,amssymb,amscd,amsfonts}
\usepackage[pagebackref=true]{hyperref}
\usepackage{url}            
\usepackage{booktabs}       
\usepackage{nicefrac}       
\usepackage{microtype}      

\usepackage{graphicx,float,latexsym,color}
\usepackage[font={scriptsize,it}]{caption}
\usepackage{subcaption}

\usepackage{makecell}

\usepackage[dvipsnames]{xcolor}

\newtheorem{theorem}{Theorem}

\newtheorem{proposition}{Proposition}

\newtheorem{remark}{Remark}
\newtheorem{corollary}{Corollary}

\newtheorem{lemma}{Lemma}

\hypersetup{
    pdftoolbar=true,        
    pdfmenubar=true,        
    pdffitwindow=false,     
    pdfstartview={FitH},    
    colorlinks=true,       
    linkcolor=OliveGreen,          
    citecolor=blue,        
    filecolor=black,      
    urlcolor=red           
}

\usepackage{lineno}

\title[Poristic Triangles and 3-Periodics]{Related by Similiarity: Poristic Triangles and 3-Periodics in the Elliptic Billiard}

\author{Ronaldo Garcia}
\address{Ronaldo Garcia,
Inst. de Matemática e Estatística,\\
Univ. Federal de Goiás,\\
Goiânia, GO, Brazil}
\email{ragarcia@ufg.br}

\author{Dan Reznik}
\address{Dan Reznik,
Data Science Consulting,\\
Rio de Janeiro, RJ, Brazil}
\email{dan@dat-sci.com}

\begin{document}
\maketitle
\begin{abstract}
Discovered by William Chapple in 1746, the Poristic family is a set of variable-perimeter triangles with common Incircle and Circumcircle. By definition, the family has constant Inradius-to-Circumradius ratio. Interestingly, this invariance also holds for the family of 3-periodics in the Elliptic Billiard, though here Inradius and Circumradius are variable and perimeters are constant. Indeed, we show one family is mapped onto the other via a varying similarity transform. This implies that any scale-free quantities and invariants observed in one family must hold on the other.

\end{abstract}

\section{Introduction}
\label{sec:intro}
The {\em Poristic family} was discovered by William Chapple in 1746 and was later studied by Euler and Poncelet \cite{chapple1746-poristic,gallatly1914-geometry,odehnal2011-poristic}. It is a 1d of set of variable-perimeter triangles (blue) with fixed Incircle and Circumcircle, Figure~\ref{fig:obtuse}. By definition its Inradius-to-Circumradius $r/R$ ratio is constant. Interestingly, the same invariance holds for the family of constant-perimeter 3-periodics in the Elliptic Billiard \cite{garcia2020-new-properties,reznik2020-intelligencer}.

Our main contribution is to show that one family is mapped onto the other via a (varying) similarity transform whose parameters we derive explicitly. Therefore, all scale-free (e.g., area and length ratios) quantities and invariants verified for one family must hold for the other. To show this we study the dimensions (semi-axes and focal lengths) of various circum- and inconics derived from the family remain constant.

\textbf{Summary of the Paper}: we start with preliminaries in Section~\ref{sec:preliminaries} and then present the following results:

\begin{itemize}
    \item Theorem~\ref{thm:caustic}: The Excentral Caustic is the MacBeath Circumconic to the Excentrals.
    \item Theorem~\ref{thm:i3p}: The moving Inconic to the Excentral centered on its Circumcenter has invariant axes. A proof appears in Appendix~\ref{app:i3p-proof}.
   \item Theorem~\ref{thm:e1}: The Incenter-centered Circumconic has invariant axes identical to the former. A proof appears in Appendix~\ref{app:e1-proof}.
   \item Theorem~\ref{thm:similarity}: Poristic and Elliptic Billiard Triangle families are related by a varying similarity transform.
\end{itemize}

\begin{figure}
    \centering
    \includegraphics[width=\textwidth]{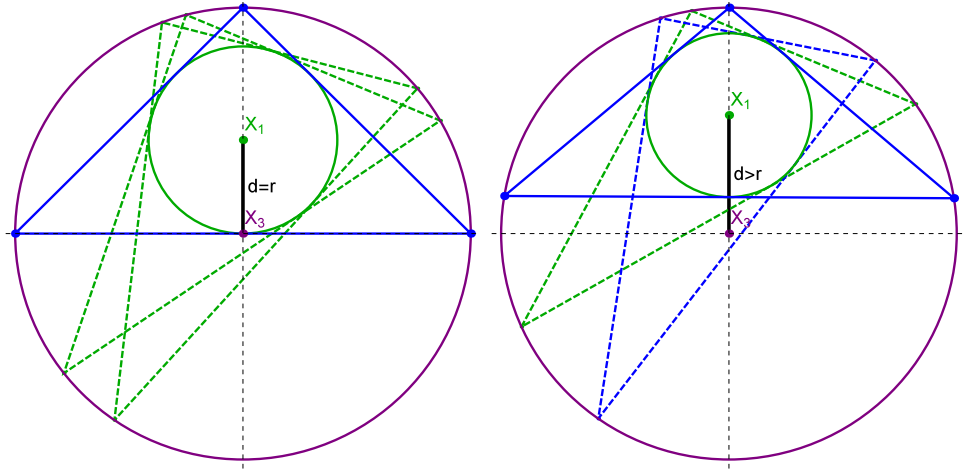}
    \caption{Poristic Triangle family (blue): fixed Incircle (green) and Circumcircle (purple). \textbf{Left}: With $d=r$, $r/R=\sqrt{2}-1$, all Poristic triangles are acute (dashed green, $X_3$ is interior) except for the one shown (blue) which is a right-triangle.  If $d<r$ (not shown), $X_3$ will lie in the Incircle (green) and the whole family is acute. \textbf{Right}: If $d>r$, $X_3$ can be both interior or exterior to the triangle, and the family will contain both acute (dashed green) and obtuse (dashed blue) triangles. \textbf{Video}: \cite[PL\#01]{reznik2020-playlist-poristic}.}
    \label{fig:obtuse}
\end{figure}

All figures reference illustrative videos in the format \cite[PL\#nn]{reznik2020-playlist-poristic}, where ``nn'' stands for the position within a playlist. For convenience, all videos mentioned are compiled on Table~\ref{tab:playlist} in Section~\ref{sec:conclusion}. Table~\ref{tab:symbols} in Appendix~\ref{app:symbols} lists all symbols used below.

\subsection{Related Work}
Weaver \cite{weaver1927-poristic} proved the Antiorthic Axis\footnote{The line passing through the intersections of reference and Excentral sidelines \cite{mw}.} of this family is stationary. In Appendix~\ref{app:weaver} we revisit and add to some of Weaver's original work \cite{weaver1927-poristic}. 

Odehnal showed the locus of the Excenters is a circle centered on $X_{40}$ and of radius $2R$ \cite{odehnal2011-poristic}. He also lists several Triangle centers which are stationary. Both circular and elliptic loci are described for Triangle Centers and vertices of derived Triangles. For example, the locus of the Mittenpunkt $X_9$ is a circle, of known radius and center \cite[page 17]{odehnal2011-poristic}; the locus of the vertices of the Tangential Triangle is an ellipse, etc.

We previously studied and Circum- and Inconics associated with the family of 3-periodics in the Elliptic Billiard \cite{reznik2020-circumconics,reznik2020-circumbilliard}, identifying certain semi-axes and focal length ratios to be invariant.

\section{Preliminaries}
\label{sec:preliminaries}
Given a generic triangle, let $d=|X_1X_3|$. Let the origin be placed at $X_3$, with $x$ running along $X_1X_3$ and $y$ along $(X_3-X_1)^t$. Note: in all of our figures, for compactness, $x$ is shown vertical. First proven by Chapple \cite{chapple1746-poristic} though known as a theorem by Euler is the relation:

\begin{equation}
    d=\sqrt{R(R-2r)}
    \label{eqn:euler}
\end{equation}

\noindent Let $\rho$ denote the invariant ratio $r/R$. For $d$ to be real in \eqref{eqn:euler}, $R/r{\geq}2$, i.e.:

\[\rho=r/R\in(0,1/2] \]

\begin{figure}
    \centering
    \includegraphics[width=\textwidth]{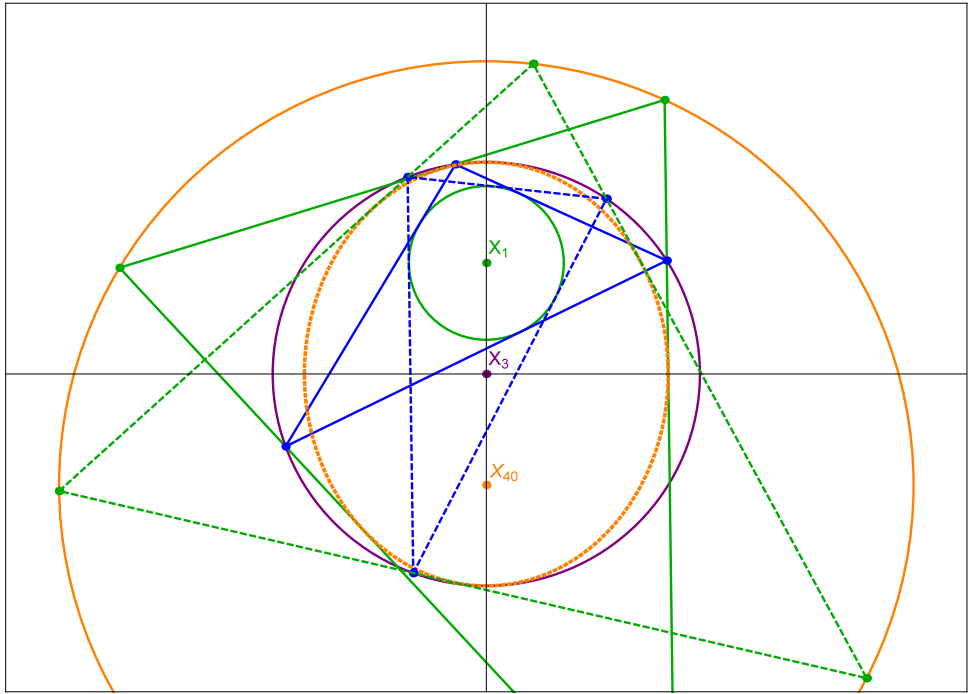}
    \caption{The centers of the Incircle (solid green) and Circumcircle (purple) are $X_1$ and $X_3$, respectively. A Poristic triangle (solid blue) and its excentral (solid green) with $r/R{\simeq}0.3266$. The same are shown (dashed) at a distinct configuration. Odehnal observed the locus of the Excenters is a circle (orange) centered on $X_{40}$ with radius $2R$ \cite{odehnal2011-poristic}. Also shown (dotted orange) is the Caustic to the Excentrals, which is the MacBeath Inconic with centers and foci at $X_i,i=5,4,3$ of the Excentrals and $X_j,j=3,1,40$ of the Poristics. Notice $X_{40}$ is the reflection of $X_1$ about $X_3$. \textbf{Video}: \cite[PL\#02]{reznik2020-playlist-poristic}.}
    \label{fig:odehnal}
\end{figure}

\begin{proposition}
The Poristic family will contain obtuse triangles if $d>r$.
\end{proposition}

This stems from the fact that when $d<r$, $X_3$ is always within the Poristic triangles, Figure~\ref{fig:obtuse}.

\section{Conic Invariants}
\label{sec:conic-invs}
Some of the results in this section were obtained with the aid of a Computer Algebra System (CAS).

\subsection{Excentral Caustic}
\label{sec:exc-i5}

\noindent Let $I_5'$ be the MacBeath Inconic \cite{mw} to the Excentral Triangles, with center and foci on the latter's\footnote{$X_i'$ refers to Triangle Centers of the Excentral Triangle.} $X_5',X_4',X_3'$, i.e., $X_3,X_1,X_{40}$, of the Poristic family, Figure~\ref{fig:odehnal}.

Let $\mu_5'$ and $\nu_5'$ be the major and minor semiaxes of $I_5'$.

\begin{theorem}
$\mu_5'=R$ and $\nu_5'=\sqrt{R^2-d^2}$ are invariant and $I_5'$ is stationary, i.e., it is the Caustic to the family of Excentral Triangles.
\label{thm:caustic}
\end{theorem}

\begin{proof} The sides of the excentral triangle $\ell_i',i=1,2,3$ are defined in \eqref{eq:r123_lines}. Observe   the translation $(d,0)$ in the parametrization of the vertices $P_i(t), i=1,2,3$ given by equation \eqref{eq:p1p2p3t} and so $X_3=(d,0)$.
It is straightforward to verify these are tangent to the ellipse:
\[
\frac{(x-d)^2}{R^2}+\frac{y^2}{R^2-d^2}=1\]
with center $X_3=(d,0)$ and foci $X_{40}=(0,0)$ and $X_1=(2d,0).$
\end{proof}

\noindent Applying \eqref{eqn:euler} to $\mu'_5/\nu_5'=R/\sqrt{R^2-d^2}$ obtain:

\begin{corollary}
The aspect ratio of $I_5'$ is given by:
\begin{equation*}
 \frac{\mu'_5}{\nu_5'}=\frac{1}{\sqrt{2 \rho}}
\end{equation*}
\end{corollary}

\noindent Let $C'$ be the circle centered on $X_3$ and of radius $2R$. Let $I_5$ be the ellipse centered n $X_5$ and with foci on $X_4$ and $X_3$.

\begin{corollary}
The conic pair ($C',I_5)$ is associated with a $N=3$ Poncelet family with stationary $X_5$.
\end{corollary}

This stems from the fact that this pair is the Excentral and Caustic to the Poristic family (taken as reference triangles).

\subsection{Excentral $X_3$-Centered Inconic}
\label{sec:exc-i3}

Let $I_3'$ be the Inconic to the Excentral Triangles centered on their stationary $X_3$ ($X_{40}$ of the Poristic family).

\noindent Let $\mu_3'$ and $\nu_3'$ be the major and minor semiaxes of $I_3'$.

\begin{theorem}
$\mu_3'=R+d$ and $\nu_3'=R-d$ are invariant over the Poristic family, i.e., $I_3'$ rigidly rotates about $X_{40}$.
\label{thm:i3p}
\end{theorem}

\begin{proof} See Appendix \ref{app:i3p-proof}.
\end{proof}



\noindent As before, applying \eqref{eqn:euler} to $\mu'_3/\nu_3'=(R+d)/(R-d)$ yields:

\begin{corollary}
The aspect ratio of $I_3'$ is invariant and given by:
\begin{equation*}
\frac{\mu_3'}{\nu_3'}= \frac{1+\sqrt{1-2\rho}}{\rho}-1
\end{equation*}
\end{corollary}

\begin{proposition}
The non-concentric conic pair ($C',I_3')$ is associated with a $N=3$ Poncelet family with stationary $X_3$.
\end{proposition}

\begin{figure}
     \centering
     \includegraphics[width=\textwidth]{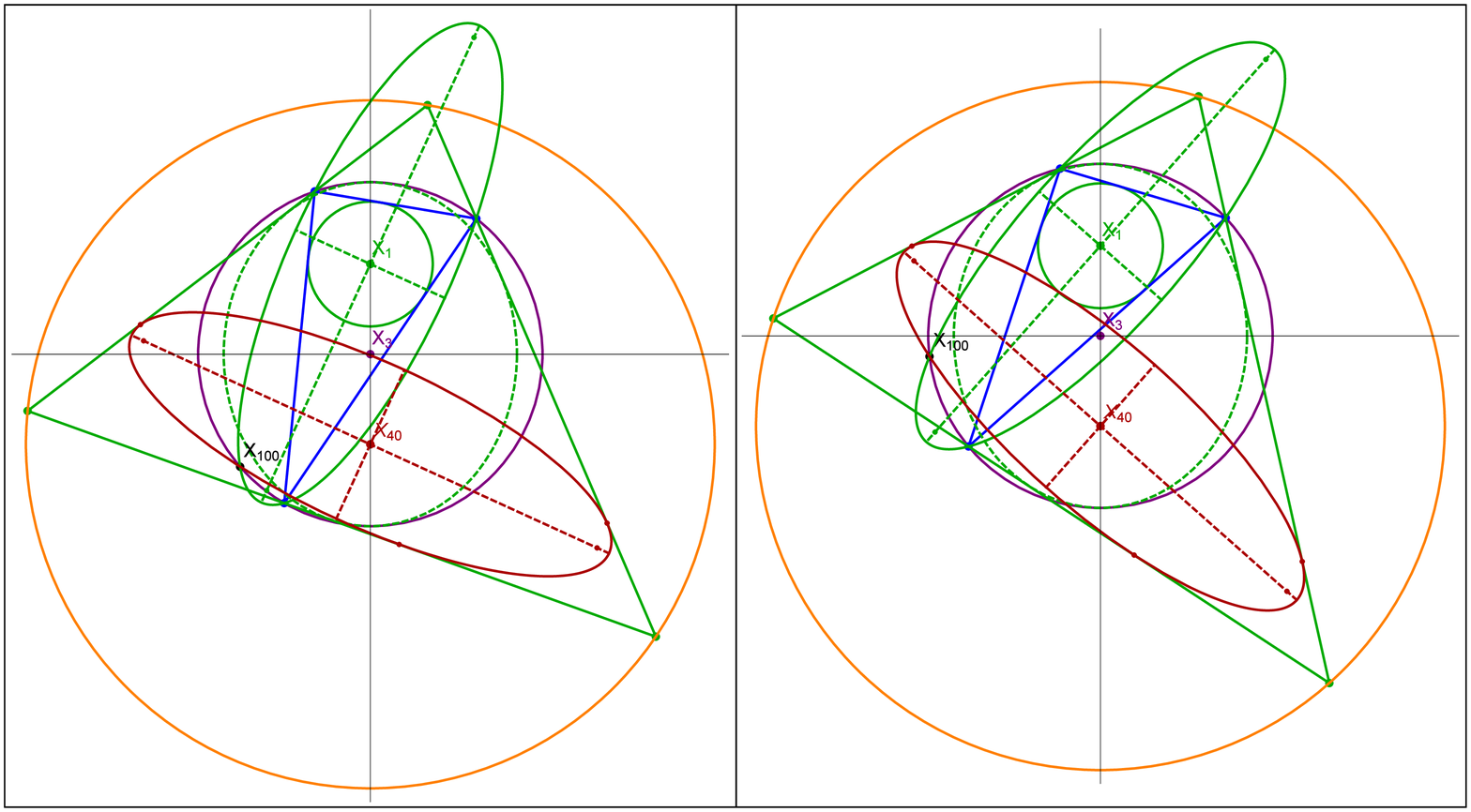}
     \caption{Inconic Invariants: two configurations shown of the Poristic Triangle family (blue). The Incircle (green) and Circumcircle (purple) are fixed, and $r/R=0.3627$. The Excentral Caustic $I_5'$ (dashed green) is the (stationary) MacBeath Inconic with center and foci at $X_i=5,4,3$ of the Excentral, i.e., $X_j,j=3,1,40$ of the Poristic (blue) triangles. The ratio of its semi-axes $\mu_5'/\mu_5= 1/\sqrt{2\rho}$. Also shown is $I_3'$, the Inconic to the Excentrals centered on its $X_3$, i.e., $X_{40}$ of the Poristic family (one of the foci of $I_5'$). Over the family, its semiaxes are invariant at $R+d$ and $R-d$, i.e., this is a rigidly-rotating ellipse about $X_{40}$. Also shown is $E_1$ (green ellipse), the $X_1$-centered circumconic, an $90^\circ$-rotated copy of $I_3'$. \textbf{Video:} \cite[PL\#03]{reznik2020-playlist-poristic}}
     \label{fig:inconic-invariants}
\end{figure}

\noindent This fact was made originally in \cite{reznik2020-circumconics}:

\begin{remark}
$I_3'$ contains $X_{100}$.
\end{remark}

\subsection{The $X_1$-Centered Circumconic}

Let $E_1$ be the Circumconic the Poristic triangles centered on $X_1$.


\noindent Let $\eta_1$ and $\zeta_1$ be the major and minor semiaxes of $E_1$.

\begin{theorem}
$\eta_1=R+d$ and $\zeta_1=R-d$ are invariant over the Poristic family, i.e., $E_1$ rigidly rotates about $X_{1}$. 
\label{thm:e1}
\end{theorem}

\begin{proof} See Appendix \ref{app:e1-proof}

\end{proof}

\begin{corollary}
The aspect ratio of $E_1$ is invariant and identical to the aspect ratio of $I_3'$. 
\end{corollary}

\begin{proposition}
$E_1$ contains $X_{100}$.
\end{proposition}

\begin{proof}
$E_1$ is the set of trilinear triples $p:q:r$ such that:
 \[E_1: (s_2+s_3-s_1)/p+(s_1+s_3-s_2)/q+(s_1+s_2-s_3)/r=0.\]
 In trilinear coordinates $X_{100}=[1/(s_2-s_3): 1/(s_3-s_1):1/(s_1-s_2)]$ and so $E_1(X_{100})=0.$
\end{proof}

Two configurations for $I_5'$, $E_1$, $I_3'$ are shown in Figure~\ref{fig:inconic-invariants}.

\subsection{$X_{10}$-circumconic}

Let $\eta_{10},\zeta_{10}$ be the major, minor semi-axes of $E_{10}$, the $X_{10}$-centered Circumconic. The locus of $X_{10}$ over the Poristic family is a circle centered on $X_{1385}$ with radius $R/4-r/2$ \cite[page 56]{odehnal2011-poristic}. Let $\eta_5'$ and $\zeta_5'$ be the major, minor semi-axes of $E_5'$, the Circumconic to the Excentrals centered on their $X_5$ (i.e., $X_3$ of the Poristics). Referring to Figure~\ref{fig:circumX10}:

\begin{proposition}
$\eta_{10}/\zeta_{10}$ is invariant and equal to $\eta_5'/\zeta_5'$. These are given by:

$$\frac{\eta_{5}'}{\zeta_{5}'}=\frac{\eta_{10}}{\zeta_{10}}=\sqrt{\frac{R+d}{ R-d}}.$$
\end{proposition}

\begin{proof}
We used a similar approach: generate candidate ratio at isosceles configuration and verify with CAS the ratio is independent of $t$.
\end{proof}

\begin{figure}
    \centering
    \includegraphics[width=.8\textwidth]{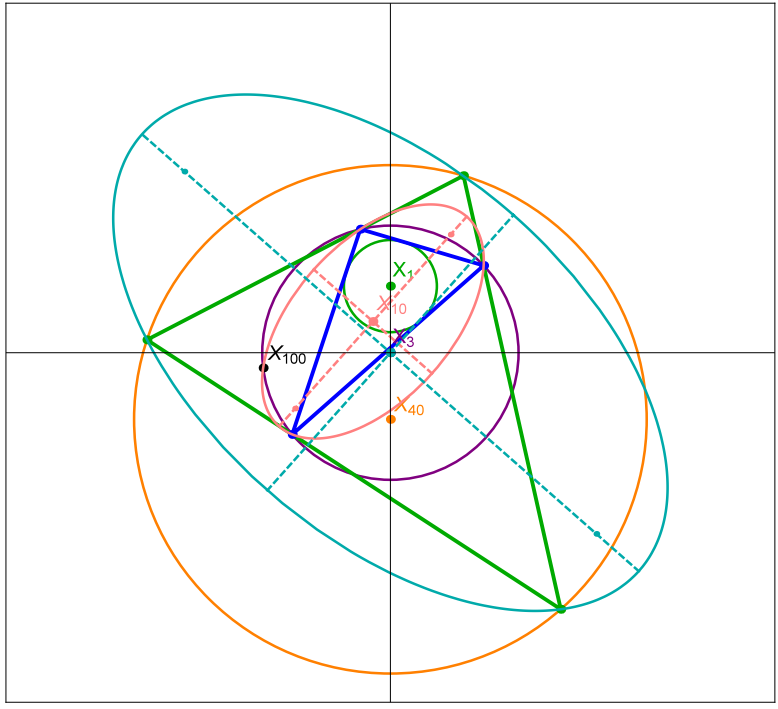}
    \caption{The Circumconic $E_{10}$ (pink) to the Poristic triangles (blue) is centered on the Spieker Center $X_{10}$. Its aspect ratio is invariant over the Poristic family and equal to that of the Circumconic to the Excentral $E_5'$ (light blue), centered on its $X_5$ ($X_3$ of the Poristics). \textbf{Video}: \cite[PL\#04]{reznik2020-playlist-poristic}}. 
    \label{fig:circumX10}
\end{figure}

\section{Connection with Elliptic Billiards}
\label{sec:cb}
The Circumbilliard to a generic triangle is the Circumconic which renders the triangle a 3-periodic orbit, i.e., it will be centered on $X_9$ \cite{reznik2020-circumbilliard}. Consider the Circumbilliard of a Poristic triangle, Figure~\ref{fig:cb-focus-locus} and let its semiaxes be denoted by $a_9,b_9$.

\begin{figure}
    \centering
    \includegraphics[width=.8\textwidth]{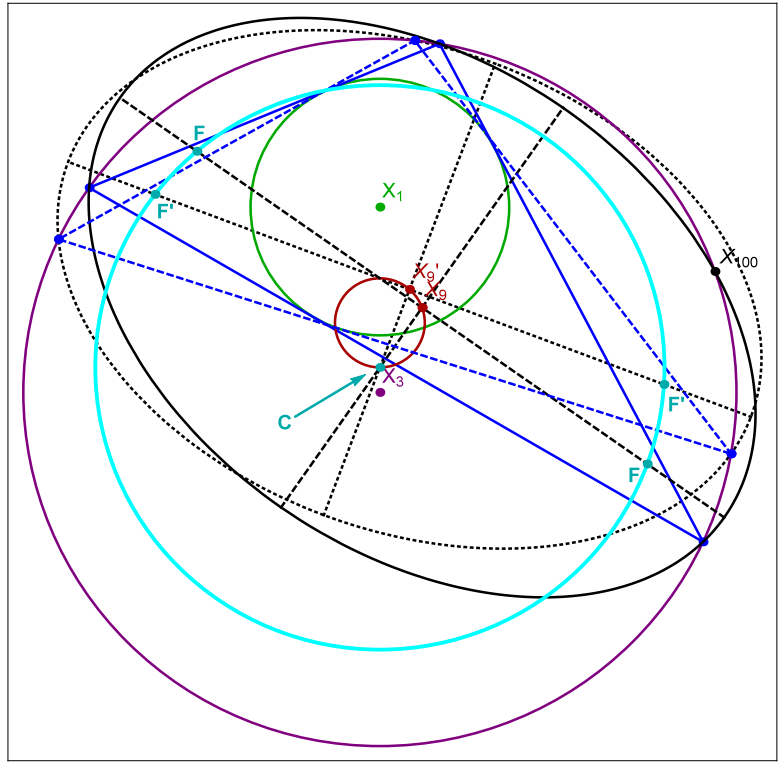}
    \caption{Two Poristic triangles (blue and dashed blue) are shown. Also shown are their Circumbilliards (black and dotted black), centered on $X_9(t)$. The locus of $X_9$ is a circle (red) \cite{odehnal2011-poristic}. It turns out the locus of the CB foci $F$ (cyan) is also a circle centered at $C$ and of radius $r_9$ (see Proposition~\ref{prop:r9}). $F'$ denote the foci of the CB of the second (dashed blue) Poristic triangle. \textbf{Video}: \cite[PL\#05]{reznik2020-playlist-poristic}.}
    \label{fig:cb-focus-locus}
\end{figure}

\begin{figure}
    \centering
    \includegraphics[width=\textwidth]{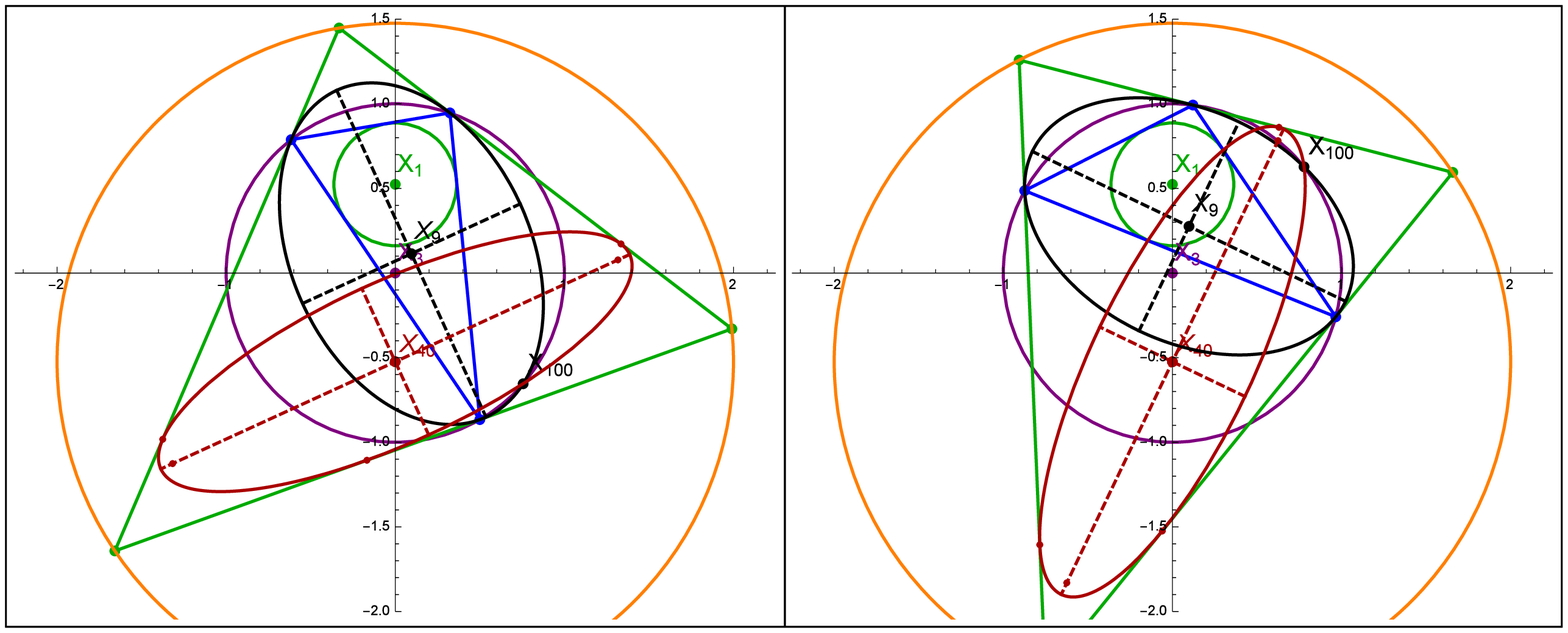}
    \caption{Two configurations (left and right) of the Poristic family (blue) for $R=1,r=0.36266$. The Incircle and Circumcircle appear green and purple. The  Excentral Triangle (green) is shown inscribed in the circular locus (orange) of its vertices \cite{odehnal2011-poristic}. Also shown is $I_3'$ (red, inconic to the Excentrals centered on its Circumcenter) and $E_9$, the Circumbilliard to the Poristic triangles (black). Over the family, (i) $E_9,I_3',E_1$ (latter not shown) have invariant aspect ratios, with the latter two identical; (ii) their axes remain parallel; (iii) all meet the Circumcircle at $X_{100}$.
    \textbf{Video}: \cite[PL\#06,07]{reznik2020-playlist-poristic}.}
    \label{fig:cb-poristic}
\end{figure}

\begin{proposition}
The perimeter $L(t)$ of a Poristic triangle is given by:

\begin{equation*}
L(t)=\frac {\left(3\,{R}^{2}
-4\,dR\cos t  +{d}^{2} \right)\sqrt{3\,{R}^{2}+2\,dR\cos t  -{d}^{2}}  }{R\sqrt {{R}^{2}-2\,dR
\cos t  +{d}^{2}}}\\
\end{equation*}
\end{proposition}
\begin{proof}
Follows directly computing $L(t)=|P_1-P_2|+|P_2-P_3|+|P_3-P_1|$ using equation \eqref{eq:p1p2p3t} and the relation $r=(R^2-d^2)/2R$. The long expressions involving square roots were manipulated using a CAS.  
\end{proof}

\begin{theorem}
The 3-periodic family is the image of the Poristic family under a one-dimensional family of similarity transformations (rigid rotation, translation, and uniform dilation). 
\label{thm:similarity}
\end{theorem}

\begin{proof} Let  $\Delta(t)=\{P_1(t),P_2(t),P_3(t)\}$ be a Poristic triangle  given by \eqref{eq:p1p2p3t} translated by $(-d,0)$ and consider the circumellipse $E_9(t)$ centered on $X_9(t)=(x_9(t),y_9(t))$ with $a_9(t)$ and $b_9(t)$ the major, minor semiaxes. 

Odehnal showed that the locus of the Mittenpunkt $X_9$ is a circle whose radius is $R{d^2}R/(9R^2-d^2)$ and center is $X_1 + (X_1 - X_3) (2 R - r)/(4 R + r)=d (3 R^2+d^2)/(9R^2-d^2) $ \cite[page 17]{odehnal2011-poristic}. 
In fact,   using the characterization of $X_9$ as the intersection of lines passing through the vertices of the excentral triangle and the medium points of the triangle $\Delta(t)$, it follows that $X_9(t)$ is parametrized by:
\begin{align*}
\begin{array}{ll}
    X_9(t)= &\left[  \frac {d \left( 4\, d \cos^2  t   
 \left( R\cos t-d \right) -r \left( 3\,d\cos t  +R \right) -{r}^{2} \right) }{ \left( 4\,R+r \right) 
 \left(d \cos t  -R+r \right) }
  , \, \frac {4R{d}^{2}\sin t  \left( R^2- \left( 2\,R\cos t-d \right) ^{2}  \right) }{ \left( {R}^{2}+{d
}^{2}-2\,dR\cos t  \right)  \left( 9\,{R}^{2}-{d}^{2}
 \right) }
    \right]
    \end{array}
\end{align*}

Let
$\theta(t)$ be the angle between $a_9(t)$ and the line $X_1X_3$, Figure~\ref{fig:cb-poristic}. 
Using the vertices $P_1(t)P_2(t)P_3(t)$, translated by $(-d,0)$ and the center $X_9(t)$
we can obtain the equation of the Circumellipse $E_9(t)$. Developing the calculations it follows that the
  angle of rotation $\theta(t)$ is given by:
\begin{align*}
\begin{array}{ll}
    \tan\theta(t)=& \frac{(1-\cos t)(R+d-2R\cos t)}{(2R\cos t+R-d)\sin t}
    \end{array}
\end{align*}

Consider the following transformation:

\begin{equation*}
\begin{array}{ll} 
    x=&L(t)(\cos \theta(t) u+\sin\theta(t) v+x_9(t) )\\
    y=&L(t)(-sin\theta(t) u+\cos\theta(t) v+y_9(t)).
\end{array}
\end{equation*}



By construction, the family of Poristic triangles $\Delta(t)$ is the image of the 3-periodic family of the elliptic billiard defined by:
\begin{align*}
\begin{array}{ll} 
E(u,v)=&\frac{u^2}{a_9^2}+\frac{v^2}{b_9^2}-1=0\\
a_9=&L(t)\frac{R\sqrt {3\,{R}^{2}+2\,dR-{d}^{2}} }{9\,{R}^{2}-{d}^{2}}= L(t){\frac {\sqrt {2}\sqrt {\rho+1+\sqrt {1-2\,\rho}}}{2\,\rho+8}},\\
b_9=&  L(t)\frac{R\sqrt {R-d}}{\sqrt {3\,R+d} (3\,R-d)}=L(t){\frac {\sqrt {2}\sqrt {\rho+1-\sqrt {1-2\,\rho}}}{2\,\rho+8}}\\
  c_9=&\sqrt{a_9^2-b_9^2}=L(t)\frac{2R\sqrt{dR}}{9R^2-d^2}.
  \end{array}
\end{align*}

\noindent Therefore, the similarity transform is given by $\theta(t),X_9(t),L(t)$.
\end{proof}

\begin{figure}
    \centering
    \includegraphics[width=\textwidth]{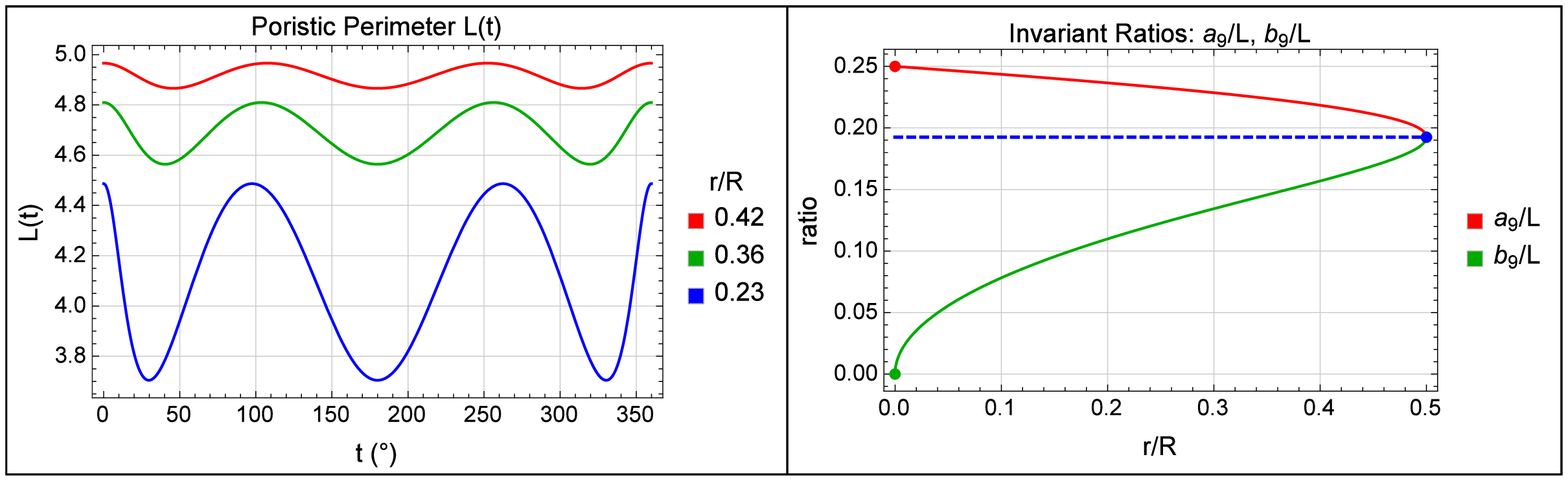}
    \caption{\textbf{Left}: Perimeter of Poristic triangles vs. parameter $t$ (of one tangent to Incircle) for various value of $r/R$; \textbf{Right}: Invariant Circumbilliard semi-axis ratios $a_9(t)/L(t),b_9(t)/L(t)$ vs $r/R\in[0,1/2]$. The dashed blue line represents their limit values of $\sqrt{3}/{9}{\simeq}0.19245$ when $r/R=1/2$ (equilateral triangles). The red and green dots show that as $r/R{\rightarrow}0$, $a_9{\rightarrow}1/4$ and $b_9{\rightarrow}0$.}
    \label{fig:cb-plots}
\end{figure}

\begin{corollary}
The ratios $a_9(t)/L(t)$, $b_9(t)/L(t)$, and $c_9(t)/L(t)$ are invariant over the Poristic family.
\end{corollary}

\begin{proposition}
The aspect ratio of the Circumbilliard is invariant over the Poristic family and given by:

\begin{equation*}
\frac{a_9(t)}{b_9(t)}=\sqrt{\frac { \left(  R+d \right)  \left( 3R-d \right) }{ \left( R-d
 \right)  \left( 3\,R+d \right) }}=\sqrt{\frac{\rho^2+2 (\rho+1)\sqrt{1-2\rho} +2}{\rho (\rho+4)}}
\end{equation*}
\end{proposition}

\begin{proof}
The following expression for $r/R$ was derived for the 3-periodic family of an $a,b$ Elliptic Billiard \cite[Equation 7]{garcia2020-new-properties}:

\begin{equation*}
  \rho = \frac{r}{R} =  \frac{2(\delta-b^2)(a^2-\delta)}{c^4}
\end{equation*}

\noindent where $\delta=\sqrt{a^4-a^2 b^2+b^4}$, and $c^2=a^2-b^2$. Solving the above for $a/b$ yields the result.
\end{proof}

Figure~\ref{fig:cb-plots} illustrates the variable perimeter and invariant aspect ratio for the CB of the Poristic family for various values of $r/R$.

\begin{corollary}
The axes of the $I_3'$ are parallel to Circumbilliard's.
\end{corollary}

This stems from the fact that $E_3'$ for 3-periodics has parallel axes to the CB \cite{reznik2020-circumconics} and the fact that it will be preserved under the similarity transform.

\begin{corollary}
The axes of the $E_{10}$ and $E'_6$ are parallel to Circumbilliard's axes.
\end{corollary}

This stems from the fact that the axes of $E_6'$ are parallel to those of $E_{10}$, and that the latter has parallel axes to the Circumbilliard \cite{reznik2020-circumbilliard}, see Figure~\ref{fig:circumX6p}.

\begin{figure}
    \centering
    \includegraphics[width=.8\textwidth]{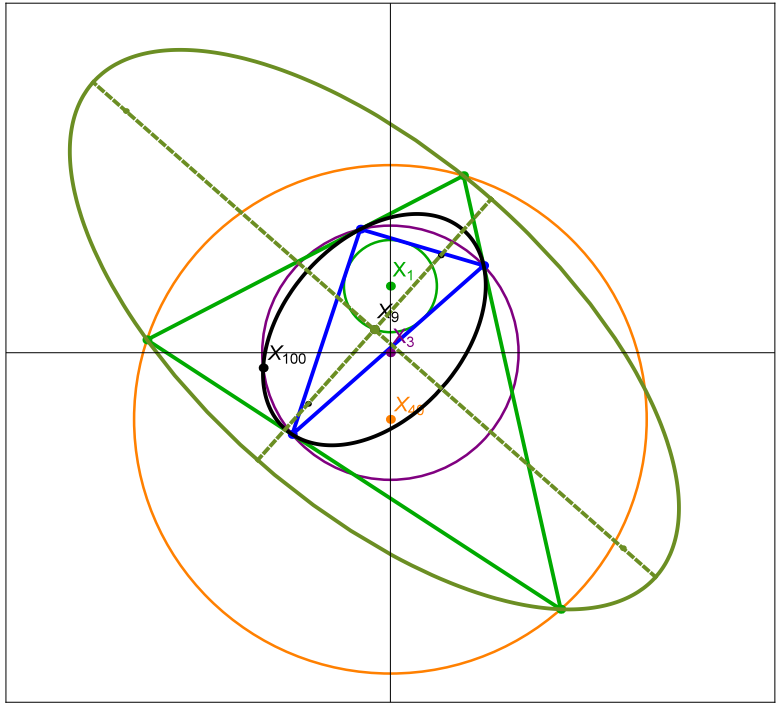}
    \caption{The Circumconic to the Excentral $E_6'$ (olive green), centered on its $X_6$ is concentric and axis-parallel to the CB (black). Both conserve their aspect ratio. The locus of the foci of the former is {\em not} an ellipse, whereas that of the CB is. \textbf{Video}: \cite[PL\#08]{reznik2020-playlist-poristic}.}
    \label{fig:circumX6p}
\end{figure}

\begin{corollary}
The aspect ratio for $I_5'$ and $I_3'$ is the invariant and the same for both Poristic and Billiard families.
\end{corollary}

\noindent This also stems from the fact that these are true for the EB \cite{reznik2020-circumbilliard} and that the aspect ratios are preserved by the similarity transform.

Let $F$ be the Feuerbach Hyperbola and $J_{exc}$ be the Excentral Jerabek Hyperbola, Figure~\ref{fig:cirumhyps}. Let their focal lengths be $\gamma$ and $\gamma'$. 

\begin{corollary}
The focal length ratio $\gamma'/\gamma=\sqrt{2/\rho}$ is invariant and the same for both Poristic and Billiard families.
\end{corollary}

Again, this ratio is invariant for 3-periodics \cite{reznik2020-circumbilliard} and must be also invariant  for the Poristic family.

\begin{figure}
    \centering
    \includegraphics[width=.8\textwidth]{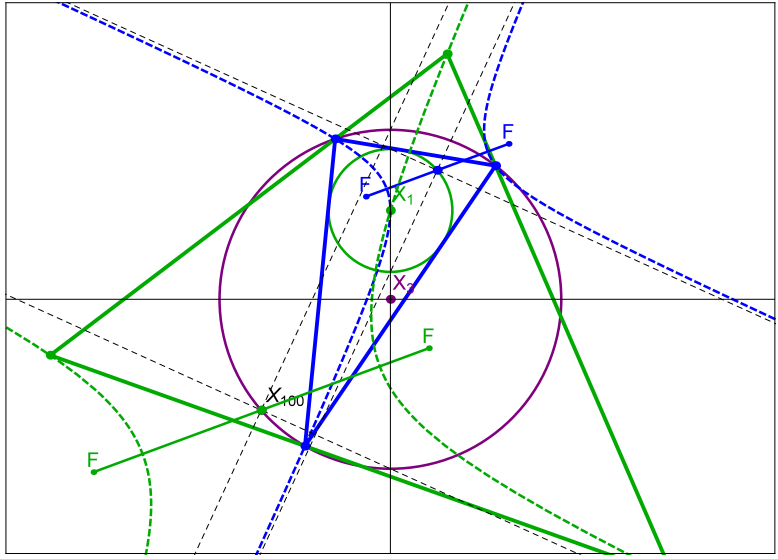}
    \caption{Feuerbach (dashed blue) and Jerabek Circumhyperbols (dashed green) to a Poristic triangle (blue) and its Excentral (green). Their asymptotes (dashed gray) are parallel (and parallel to the axes of the CB, not shown). Also shown are their foci (blue, green ``F''), and their parallel focal axes (solid blue and green). The ratio $\gamma/\gamma'$ of their focal lengths is invariant over the family. \textbf{Video:} \cite[PL\#10,11]{reznik2020-playlist-poristic}.}
    \label{fig:cirumhyps}
\end{figure}

With the aid of CAS, the following can be shown:

\begin{proposition}
Over the Poristic family, the foci of the Circumbilliard describe a circle with center
$[(R-d)d/(3R+d),0]$ and radius $r_9$ given by:

\[ r_9=\frac{4d (R-d)\sqrt{dR}}{ (3R-d) \sqrt{(3R-d)(R+d)}} \]
\label{prop:r9}
\end{proposition}

Let $\eta_6',\zeta_6'$ be the major, minor of the $E_6'$, the Circumconic to the Excentrals centered on their $X_6$ ($X_9$ of the Poristics. Referring to Figure~\ref{fig:circumX6p}:

\begin{proposition}
The $E_6'$ is concentric an has parallel axes to the Circumbilliard. Furthermore, its aspect ratio is given by:

\begin{align*}
    \frac{\eta_6'}{\zeta_6'}=\frac{b_9^2+\delta}{a_9}\frac{b_9}{a_9^2+\delta}=\sqrt{\frac{(R+d)(3R+d)}{(3R-d)(R-d)}}
\end{align*}

\noindent $\delta=\sqrt{a_9^4-{a_9^2}b_9^2+b_9^4}$.
\end{proposition}

This stems from the fact that $E_6'$ for 3-periodics is the locus of the Excenters, shown to be an ellipse with said aspect ratio \cite{garcia2019-incenter}.

\begin{figure}
    \centering
    \includegraphics[width=\textwidth]{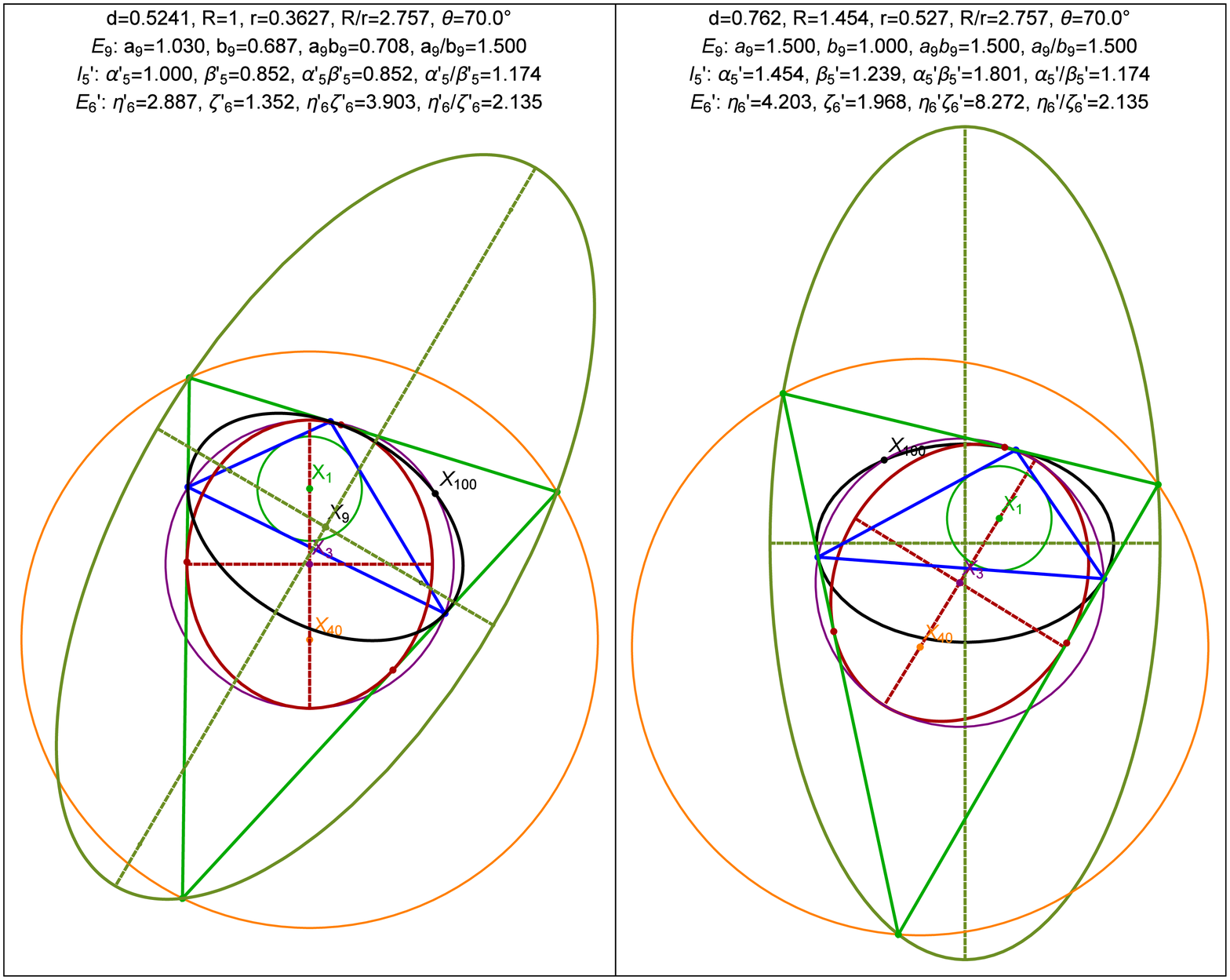}
    \caption{\textbf{Left:} Poristic triangle (blue), stationary Incircle (green) and Circumcircle (purple). Varying Poristic CB (black), whose aspect ratio is constant. Stationary Excentral MacBeath Inconic and Caustic $I_5'$ (red), circular Excentral locus (orange), and Excentral (MacBeath) Circumconic $E_6'$ (olive green), all with invariant aspect ratios. \textbf{Right:} same objects observed on a stationary Elliptic Billiard system: Incircle and Circumcircle are varying (though $r/R$ is invariant). $I_5'$ is moving though its aspect ratio is invariant and equal to its counterpart in the Poristic system. Conversely, $E_6'$ is now stationary and is the locus of the Excenters \cite{garcia2019-incenter}. Notice the Excentral Circumcircle (orange) is movable. \textbf{Video}: \cite[PL\#09]{reznik2020-playlist-poristic}}
    \label{fig:dual-circums}
\end{figure}

\section{Conclusion}
\label{sec:conclusion}
Table~\ref{tab:invariant-conics} summarizes properties and invariants for the various circum- and inconics mentioned above. A comparison between basic parameters in the Poristic family and 3-Periodics in the Elliptic Billiard appear on Table~\ref{tab:invariant-qtys}. Finally, shape invariances of conics in either family are compared on Table~\ref{tab:invariant-objs} and illustrated in Figure~\ref{fig:dual-circums}.

\begin{table}[]
\scriptsize
\begin{tabular}{|c|l|l|c|c|l|}
\hline
conic & poristic & EB & $X_{100}$ & ctr & note\\
\hline
$E_1$ & axes & ratio & y & $X_1$ & center on $F_{med}$  \\
$E_9$ & ratio & axes & y &
$X_9$ & (Circum-) EB, center on $F_{med}$ \\
$E_{10}$ & ratio & ratio & y & $X_{10}$ & center on $F_{med}$ \\
$I_9$ & ratio &  axes & -- & $X_9$ & Mandart Inellipse, EB Caustic\\
\hline
$E_3'$ & axes & ratio & y & $X_{40}$ & Excentral Circumcircle \\
$E_5'$ & ratio & ratio & -- & $X_3$ & same ratio as $E_{10}$\\
$E_6'$ & ratio & axes & -- & $X_9$ & MacBeath Circumconic \\
$I_3'$ & axes & ratio & y & $X_{40}$ & $90^\circ$-rotated copy of $E_1$ \\
$I_5'$ & axes & ratio & -- & $X_3$ & McBeath Inconic, Excentral Caustic \\
\hline
\end{tabular}
\caption{Table of conics, all with mutually parallel axes (except for $I_5'$). Columns ``poristic'' and ``EB'' define whether for that family the aspect ratio is invariant. $E_k$ (resp. $I_k$) stands for the Circumellipse (resp. Inellipse) centered on $X_k$. $E'_k,I'_k$ refer to Excentral conics.}
\label{tab:invariant-conics}
\end{table}

\begin{table}[]
\scriptsize
\begin{tabular}{|c|c|c|}
\hline
qty. & poristic & EB \\
\hline
$d$ & y & -- \\
$r$ & y & -- \\
$R$ & y & -- \\
$r/R$ & y & y \\
$R{\pm}d$ & y & -- \\
$\frac{R+d}{R-d}$ & y & y \\
$L$ & -- & y \\
$J$ & -- & y \\
\hline
\end{tabular}
\caption{Column ``poristic'' (resp. EB) indicates if the named quantity is invariant in the given family. Only $r/R$ and $(R+d)/(R-d)=f(r/R)$ are invariant on both.}
\label{tab:invariant-qtys}
\end{table}

\begin{table}[]
\scriptsize
\begin{tabular}{|l|c|l|l|l|l|}
\hline
object & ctr & semiaxes & poristic & EB & note \\
\hline
Incircle & $X_1$ & $r$ & y & -- & \\
Circumcircle & $X_3$ & $R$ & y & -- & \\
$I_5'$ & $X_3$ & $R,\sqrt{R^2-d^2}$ & y & -- & poristic exc. caustic \\
$E_6'$ & $X_9$ & $(b_9^2+\delta)/a_9,(a_9^2+\delta)/b_9$ & -- & y & EB exc locus \\
Exc. Circumcircle & $X_{40}$ & $2R$ & y & -- & \\
Elliptic Billiard & $X_9$ & $a_9,b_9$ & -- & y & \\
\hline
\end{tabular}
\caption{Various position and axes for conics in each Poristic and 3-periodic (EB) families. A ``y'' in the ``poristic'' or ``EB'' columns indicates shape invariance.}
\label{tab:invariant-objs}
\end{table}

Videos mentioned above have been placed on a \href{https://bit.ly/3aVCfSS}{playlist} \cite{reznik2020-playlist-poristic}. Table~\ref{tab:playlist} contains quick-reference links to all videos mentioned, with column ``PL\#'' providing video number within the playlist.

\begin{table}[H]
\scriptsize
\begin{tabular}{|c|l|l|}
\hline
\href{https://bit.ly/2VKuudH}{PL\#} & Title & Section\\
\hline

\href{https://youtu.be/DS4ryndDK6Q}{01} & \makecell[lt]{Poristic family, circular locus \\of excenters, and Antiorthic axis} &
\ref{sec:intro}, App. \ref{app:weaver} \\

\href{https://youtu.be/yEu2aPiJwQo}{02} & \makecell[lt]{Poristic Circumbilliard (CB) has invariant\\aspect ratio} &
\ref{sec:conic-invs} \\

\href{https://youtu.be/0VHBjdHXbJc}{03} & \makecell[lt]{$E_1$ and $I_3'$ have constant, parallel, and\\identical semi-axes} &
\ref{sec:conic-invs} \\

\href{https://youtu.be/-4AAUSFxvmo}{04} & \makecell[lt]{$E_{10}$ and $E_5'$ have axes parallel to the
Poristic CB\\as well as invariant, identical aspect ratio} &
\ref{sec:conic-invs} \\

\href{https://youtu.be/LGgh11LMGGY}{05} & \makecell[lt]{Loci of center and foci of Poristic CB are circles} &
\ref{sec:cb} \\

\href{https://youtu.be/0VHBjdHXbJc}{06} & \makecell[lt]{$I_3'$ has constant semi-axes, parallel to those\\of the Poristic CB} &
\ref{sec:cb} \\

\href{https://youtu.be/CHbrZvx1I8w}{07} & \makecell[lt]{$I_3'$ and $I_5'$ of 3-Periodics in the EB have\\invariant aspect ratio.} &
\ref{sec:cb} \\

\href{https://youtu.be/Fy4T-dmu-8s}{08} & \makecell[lt]{$E_6'$ has invariant aspect ratio and its axes\\coincide with those of the Poristic CB} &
\ref{sec:cb} \\

\href{https://youtu.be/NvjrX6XKSFw}{09} & \makecell[lt]{Side-by-side Poristic and Elliptic Billiard (EB)\\
Excentral MacBeath Inconic and Circumconics} &
\ref{sec:cb} \\

\href{https://youtu.be/QZN82WDTGGY}{10},\href{https://youtu.be/Pz4tUijYZCA}{11} & \makecell[lt]{$F$ and $J_{exc}$ Circumhyperbolas have invariant\\focal length ratio over 3-periodic family} &
\ref{sec:cb} \\
\hline
\end{tabular}
\caption{Videos mentioned in the paper. Column ``PL\#'' indicates the entry within the playlist \cite{reznik2020-playlist-poristic}}
\label{tab:playlist}
\end{table}

\section*{Acknowledgments}
\noindent We would like to thank Boris Odehnal for his proof of Theorem~\ref{thm:i3p}, and Prof. Jair Koiller for his suggestions and editorial help. The first author is fellow of CNPq and coordinator of Project PRONEX/ CNPq/ FAPEG 2017 10 26 7000 508.

\bibliographystyle{spmpsci}
\bibliography{elliptic_billiards_v3,authors_rgk_v1} 

\appendix
\section{Table of Symbols}
\label{app:symbols}
Tables~\ref{tab:kimberling} and \ref{tab:symbols} lists most Triangle Centers and symbols mentioned in the paper.

\begin{table}[H]
\scriptsize
\begin{tabular}{|c|l|l|}
\hline
Center & Meaning & Note\\
\hline
$X_1$ & Incenter & Trilinear Pole of $\mathcal{L}_1$, focus of $I_5'$ \\
$X_3$ & Circumcenter &  Focus of $I_5$ \\
$X_4$ & Orthocenter & Focus of $I_5$ \\
$X_5$ & Center of the 9-Point Circle & Center of $I_5$\\
$X_6$ & Symmedian Point & \\
$X_9$ & Mittenpunkt & Center of (Circum)billiard \\
$X_{10}$ & Spieker Point & Incenter of Medial \\
$X_{40}$ & Bevan Point & Focus of $I_5'$ \\
$X_{100}$ & Anticomplement of $X_{11}$ & Lies on $E_i,i=1,3,9,10$ and  $I_3'$ \\
$X_{650}$ & Cross-difference of $X_1,X_3$ & Generates $X_1X_3$ \\
$X_{651}$ & Isogonal Conjug. of $X_{650}$ & Trilinear Pole of $\mathcal{L}_{650}=X_1X_3$ \\
$X_{1155}$ & Schröder Point & Intersection of $X_1X_3$ with Antiorthic Axis\\
\hline
$\mathcal{L}_1$ & Antiorthic Axis & Line $X_{44}X_{513}$ \cite{etc_central_lines}\\
$\mathcal{L}_{650}$ & OI Axis & Line $X_1X_3$  \\
\hline
\end{tabular}
\caption{Kimberling Centers and Central Lines mentioned in paper}
\label{tab:kimberling}
\end{table}

\begin{table}
\scriptsize
\begin{tabular}{|c|l|l|}
\hline
Symbol & Meaning & Note\\
\hline
$P_i,s_i$ & Vertices and sidelengths of Poristic triangles & \\
$P_i'$ & Vertices of the Excentral triangle & \\
$X_i,X_i'$ & Kimberling Center $i$ of Poristic, Excentral & \\
$a_c,b_c$ & Semi-axes of confocal Caustic & \\ 
$r,R,\rho$ & Inradius, Circumradius, $r/R$ & $\rho$ is invariant \\
$d$ & Distance $|X_1{X_3}|$ & $\sqrt{R(R-2r)}$\\
$a_9,b_9$ & Semi-axes of Poristic CB & \\
$\delta$ & Constant associated w/ the CB & $\sqrt{a_9^4-a_9^2 b_9^2+b_9^4}$ \\
\hline
$E_i$ & Circumellipse centered on $X_i$ & \makecell[tl]{Axes parallel to $E_9$ if $X_i$ on $F_{med}$} \\
$E_i'$ & Excentral Circumellipse centered on $X_i'$ & \\
$\eta_i,\zeta_i$ & Major and minor semiaxis of $E_i$ & Invariant ratio for $i=1,3,9,10$ \\
$\eta_i',\zeta_i'$ & Major and minor semiaxis of $E_i'$ & Invariant ratio for $i=3,5,6$ \\
\hline
$I_i$ & Inellipse on $X_i$ & $I_3$; MacBeath $I_5$; Mandart $I_9$ \\
$I_i'$ & Excentral Inellipse centered on $X_i'$ & $I_3'$; MacBeath $I_5'$ \\
$\mu_i,\nu_i$ & Major and minor semiaxis of $I_i$ & \\
$\mu_i',\nu_i'$ & Major and minor semiaxis of $I_i'$ & Invariant ratio for $i=3,5$ \\
\hline
$F_{exc}$ & $F$ of Excentral Triangle & Center $X_{3659}$ \cite{moses2020-private-circumconic} \\
$J_{exc}$ & $J$ of Excentral Triangle & Center $X_{100}$, Perspector $X_{649}$ \\
$F_{med}$ & $F$ of Medial & Center $X_{3035}$ 
\cite{moses2020-private-circumconic} \\
$\lambda',\lambda$ & Focal lengths of $J_{exc},F$  & Invariant ratio \\
\hline
\end{tabular}
\caption{Symbols used in paper}
\label{tab:symbols}
\end{table}

\section{Weaver Invariants}
\label{app:weaver}
\subsection{Antiorthic Axis}

The Antiorthic axis $\mathcal{L}_1$ is stationary, and $X_{1155}$ is stationary intersection of $\mathcal{L}_1$ with $\mathcal{L}_{650}=X_1X_3$, Figure~\ref{fig:antiorthic}. The anthiortic axis is given by:

\[ x= \frac{3R^2-d^2}{2d}\]

\begin{figure}[H]
    \centering
    \includegraphics[width=\textwidth]{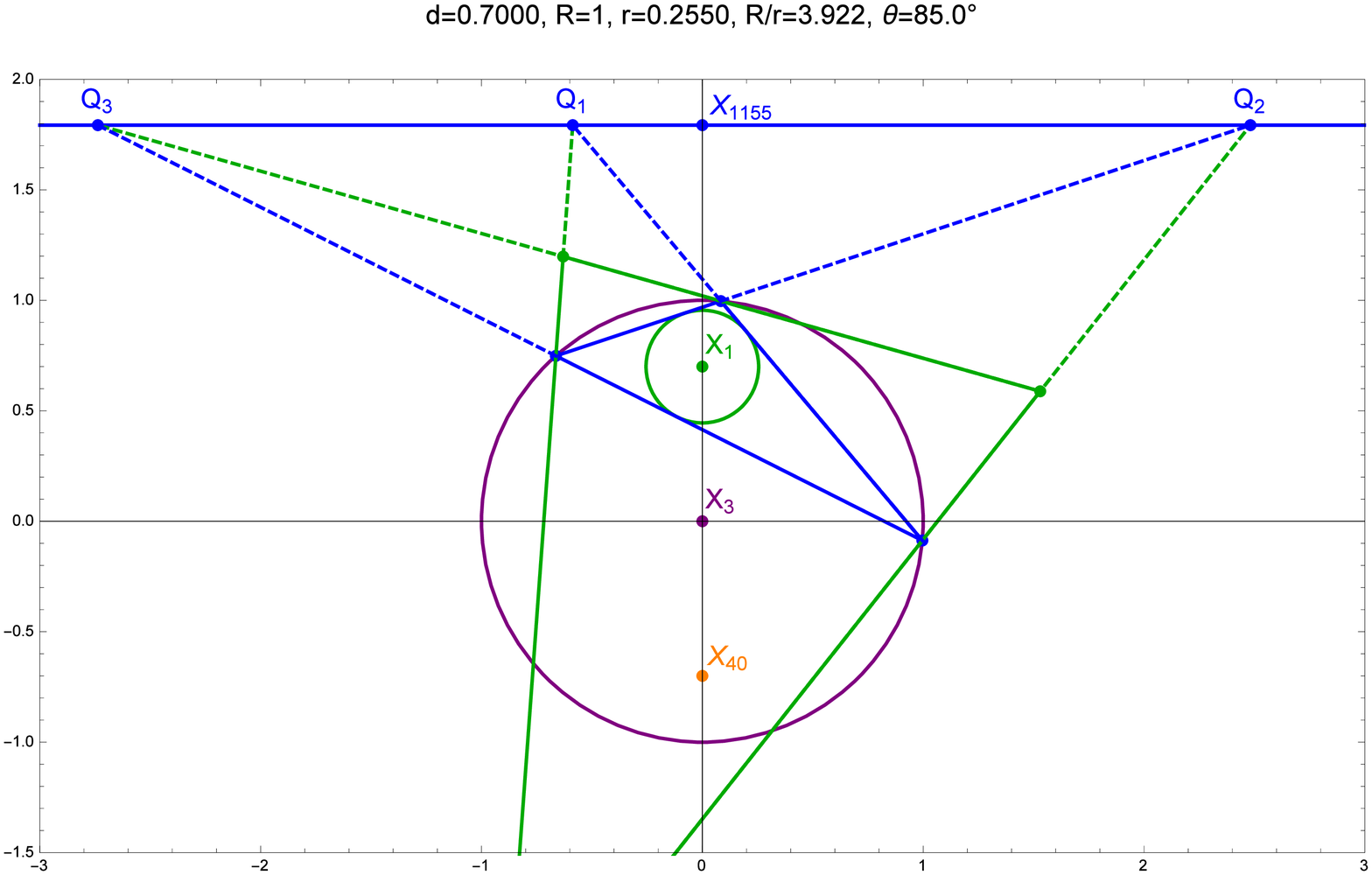}
    \caption{A result by Weaver \cite{weaver1927-poristic} is that over the Poristic family, the Antiorthic Axis $L_1$ is Invariant. Odehnal observed $X_{1155}$ was one of the many stationary Triangle centers along $L_{663}=X_1X_3$. This point happens to lie at the latter's intersection with $L_1$. \textbf{Video:} \cite[PL\#01]{reznik2020-playlist-poristic}.}
    \label{fig:antiorthic}
\end{figure}

\subsection{A possible correction to Weaver's 2nd order invariant}

Assume the origin\footnote{In Weaver's paper, the origin is on $X_1$, so the center of $C_w$ is at $[-R-d,0]$.} is on $X_3$. In \cite[Theorem III]{weaver1927-poristic} it is proposed that a circle $C_w$ centered on $[-R,0]$ and of radius $\sqrt{Rd(R+d)(R+d+r)}/d$ has the same power with respect to the Antiorthic axis $\mathcal{L}_1$ as the Incircle. We have found this not to be the case. Let $I_1:(x-d)^2+y^2=r^2$ denote the Incircle. Referring to Figure~\ref{fig:weaver-circle}(left), let $C_w'$ be circle centered on $[-R,0]$ and of radius:

\begin{equation*}
  r_w'=
  \left( \frac{d+R}{2R} \right) \sqrt{\frac{(3\,R-d)(4\,{R}^{2}-Rd-{d}^{2})}{d} }
 \end{equation*}

\begin{proposition}
$C_w'$ and $I_1$ have the same power with respect to $\mathcal{L}_1$.
\end{proposition}

\begin{proof}
Translate the vertices of the Poristic family in \eqref{eq:p1p2p3t} by $(-d,0)$. It is straightforward to show that the Antiorhtic axis $\mathcal{L}_1$ is given by $x=(3R^2-d^2)/(2d)$.
The power $P_w$ of $P_0=[(3R^2-d^2)/(2d),0]$ with respect to the circle $C_w':(x+R)^2+y^2=r_w'^2 $   is given by:

 $$ P_w(P_0,C_w')=|P_0-[-R,0]|^2-r_w'^2=\frac{(d+R)^2(3R-d)^2}{4d^2}-r_w'^2.$$ Also,
 
 $$ P_w(P_0,I_1)=|P_0-[d,0]|^2-r^2  = \frac { \left( {R}^{2}-{d}^{2} \right) ^{2} \left( 9\,{R}^{2}-{d
}^{2} \right) }{4{R}^{2}{d}^{2}}.
  $$
 Therefore,
$\mathcal{L}_1$ is the radical axis of the pair of circles $C_w'$ and $I_1$ if, and only if,
$$ \frac{(d+R)^2(3R-d)^2}{4d^2}-r_w'^2=\frac { \left( {R}^{2}-{d}^{2} \right) ^{2} \left( 9\,{R}^{2}-{d
}^{2} \right) }{4{R}^{2}{d}^{2}}.$$
 Solving the equation above leads to the result.
\end{proof}

\noindent Additionally, we derive a circle whose power with respect to $\mathcal{L_1}$ is equal to the Circumcircle's, Figure~\ref{fig:weaver-circle}(right). Let $C_w''$ be a circle centered on $[-R,0]$ and of radius:

\begin{equation*}
    r_w''=\sqrt {{\frac { \left( 3\,R-d \right)  \left( d+R  \right) R}{d}}} 
\end{equation*}

\begin{proposition}\label{prop:power_circles}
$C_w''$ has the same same power with respect to $\mathcal{L}_1$ as (i) the Circumcircle $C$, and (ii) $C_e$, centered on $X_{40}$ and of radius $2R$ (locus of the Excenters).
\end{proposition}

\begin{proof} Translate the verices of the Poristic family in \eqref{eq:p1p2p3t} by $(-d,0)$. 
Also, $\mathcal{L}_1$ is the radical axis (see \cite[Chapter 2]{coxeter67}) of the pair of circles

\[ C_e: \; (x+d)^2+y^2=4R^2, \;C:\; x^2+y^2=R^2.\]

 In fact, the power $P_w$ of $P_0=((3R^2-d^2)/(2d),0)$ with respect to the circles $C_e$ and $C$ is given by:
 $$P_w(P_0,C_e)=|P_0-(-d,0)|^2-4R^2= \frac{Rr(4R+r)}{R-2r}= P_w(P_0,C).$$
  
\noindent Consider the pair of circles
 \[C_w'': (x+R)^2+y^2=r_w^2, \;C:\; x^2+y^2=R^2\]
  Analogously, $P_w(P_0,C_w')=P_w(P_0,C)$ if, and only if,
   $r_w^2=    \left( 3\,R-d \right)  \left( d+R  \right) R/d.$
 
\end{proof}

\begin{figure}[H]
    \centering
    \includegraphics[width=\textwidth]{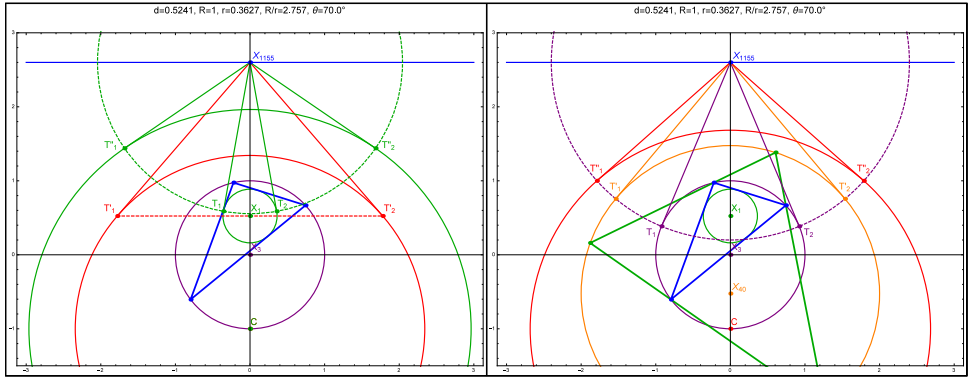}
    \caption{\textbf{Left:} Circle $C_w$ proposed in \cite[Theorem III]{weaver1927-poristic} (red) does not have the same power as the Incircle $I_1$ (green) with respect to $\mathcal{L}_1$ (blue): rather, its tangency points $T_1',T_2'$ from from $X_{1155}$ are collinear with $X_1$. We derived a new equal power circle $C_w'$ (green) of radius $r_w'$ (see text): its tangency points $T_1'',T_2''$ are concyclic with $T_1,T_2$. Note: both $C_w,C_w'$ are centered on $C=[-R,0]$. \textbf{Right}: the following three circles have the same power with respect to $\mathcal{L}_1$: (i) the Circumcircle $C$, (ii) $C_e$, the $X_{40}$-centered circular locus of the Excenters of radius $2R$ (orange), and (iii) $C_w''$, centered on $[-R,0]$ and of radius $r_w''$ (red), see text. Notice tangency point $T_i,T_i',T_i'',i=1,2$ are concyclic.}
    \label{fig:weaver-circle}
\end{figure}


\section{$I_3'$ Axis Invariance}
\label{app:i3p-proof}
Here we reproduce a proof that the axes of $I_3'$ are invariant and equal to $R{\pm}d$, kindly contributed by Odehnal \cite{odehnal2020-private-i3}.

Let $X_{40}$ be the origin and the x-axis run along $X_3X_1$, Figure~\ref{fig:odehnal-coords}. Parametrize Poristic triangles $\Delta(t)=P_1P_2P_3$ by their tangency point on the Incircle \cite{odehnal2011-poristic}:

\begin{figure}[H]
    \centering
    \includegraphics[width=.5\textwidth]{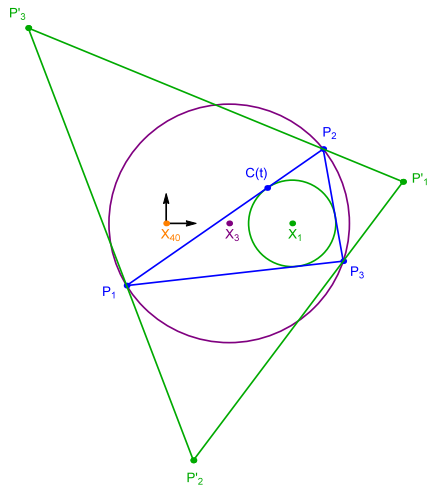}
    \caption{Coordinate System used in this Appendix: the origin is on $X_{40}$ and the positive x-axis runs along $X_3X_1$. The Poristic family is parametrized by a first tangency point $C(t)$ on the Incircle.}
    \label{fig:odehnal-coords}
\end{figure}

\begin{align} \label{eq:p1p2p3t}
\begin{array}{ll}
    P_1=&\left[ \cos t(d \cos t  +r) - \omega
\sin t  +d, ( d\cos t  +r ) \sin t +\omega\cos t\right]\\
P_2=& \left[ \cos t (d \cos t +r) +
\omega\sin t+d, ( d\cos t +r) \sin t - \omega\cos t \right]
 \\
P_3=& \left[\frac{R ( 2\,dR - ( {R}^{2}+{d}^{2} ) \cos t) }{R^2-2\,dR\cos t+d^2}+d,
\frac{ R(d^2-{R}^{2})  \sin t }{R^2-2\,d\,R
\cos t+{d}^{2}}\right]\\
\omega=&\sqrt{R^2-(d\cos t +r)^2 }
\end{array}
\end{align}

\begin{lemma}\label{lem:inconica}
  Let $ \ell_{i}: a_i x + b_i y + c_i=0 $ {\text (}$ i\in \{1, 2, 3\} $\text{)} be three tangents lines
of a conic  $E : Ax^2 + 2Bxy + Cy^2 + D = 0 $ centered at $ (0, 0)$. Then, the
coefficients $ A,B,C,D $  are given by:

\begin{align*}
\begin{array}{ll}
A=& a_2a_3 c_1^2\delta_{23} -a_1 a_3c_2^2\delta_{13} +a_1 a_2 c_3^2\delta_{12}\\
B=&\frac{1}{2}\left( (a_2 b_3+a_3 b_2) c_1^2\delta_{23} -(a_1 b_3+a_3 b_1) c_2^2  \delta_{13}+ (a_1 b_2+a_2 b_1) c_3^2\delta_{12}  \right)\\
C=&b_2 b_3  c_1^2 \delta_{23} -b_1 b_3  c_2^2\delta_{13} +b_1 b_2c_3^2\delta_{12}\\
D=& \frac{1}{4}(\delta_{12}\delta_{13}\delta_{23})^{-1}
\left( \delta_{23} c_1+\delta_{13} c_2-\delta_{12} c_3\right) 
\left(   \delta_{23} c_1-\delta_{13} c_2-\delta_{12} c_3 \right)\\
&\left( \delta_{23} c_1-\delta_{13}c_2+\delta_{12} c_3 \right) 
\left( \delta_{23}c_1+\delta_{13}c_2+\delta_{12}c_3 \right) 
\end{array}
\end{align*}
Here $\delta_{ij}=a_i b_j - a_jb_i$.

\end{lemma}

\begin{proof} 
	The condition of tangency of $\ell_i$ ($ i=1,2,3$) with $E$  is given by the discriminant equation
	\[ (AC-B^2)c_i^2+( Ab_i^2-2B a_ib_i+Ca_i^2)D=0.\]
	Solving the system leads to the result stated.
\end{proof}

\begin{lemma}
\label{lem:tan_inconica}
The three sides of the excentral triangle $\Delta'(t)=\{ P_1'(t), P_2'(t), P_3'(t)\}$ are given by the straight lines
{\scriptsize  
	\begin{align}\label{eq:r123_lines}
	\begin{array}{ll}
     \ell_1'(t):&  ( (d\sin t-\omega)\sin t-r \cos t)x-(( d\cos t +r)\sin t-\omega\cos t )y+R^2-d^2=0\\
\ell_{2}'(t):&\left((  d\sin t+\omega) \sin t  -r\cos t       \right) x- \left(   (d  \cos t +r)\sin t+ \omega\cos t \right) y+{R}^{2}-{d}^{2}=0\\
	\ell_{3}'(t):& \left(R\cos t -d\right) x+R \sin t \,  y-
	2\,dR\cos t  +{R}^{2}+{d}^{2}=0.
\end{array}
	\end{align}
	}

\end{lemma}

\begin{proof} Direct calculations of the external bisector lines passing through the vertices $P_1(t)$, $P_2(t)$ and $P_3(t)$ given by equation \eqref{eq:p1p2p3t}.
\end{proof}

\begin{proposition}
$I_3'$ is given implicitly by the equation:
{ \small 	
	\begin{align}
	\label{eqn:I3t}
	\begin{array}{lll}
	 {I}_3'(x,y,t) &=&((R^2-d^2)^2-8 dR^2   (R \cos t-d) \sin^2t) x^2\\
	&+&((R^2-d^2)^2-4d R \cos t   ((R \cos t-d)^2-R^2 \sin^2t)) y^2\\
	&+&4dR \sin t  (2 R \cos t-R-d) (2 R \cos t+R-d) xy\\
	&-&(R^2-d^2)^2 ( R^2+d^2-2 dR \cos t )=0
	\end{array}
	\end{align}
	}
\end{proposition}

\begin{proof}
Consider the Poristic defined by the circles $(x-d))^2+y^2=R^2 $and $(x-2d)^2+y^2=r^2$.
	By \cite{odehnal2011-poristic} we know that  the inconic $I_3'(t)$, tangent to the sides of the excentral triangle $\Delta'(t)$,   is centered in $X_{40}=(0,0)$. 
	Applying  lemmas \ref{lem:inconica} and \ref{lem:tan_inconica}, and using the Euler relation $R^2-d^2=2rR$, the equation \eqref{eqn:I3t} is obtained. This expression was confirmed using CAS.
	
Consider a rotation of the coordinates of \eqref{eqn:I3t} by angle $\theta$ defined by:

	\[\tan 2\theta={\frac {\sin t  (R^2 - \left( 2\,R\cos t-d)^{2}  \right) }{\cos t
 \left(  ( 2\,R\cos t-d) ^{2}-3\,{R}^{2}
 \right) +2\, dR}}.\]

This re-expresses \eqref{eqn:I3t} in canonical form:
	
\begin{equation}
 ( R^2+d^2- 2dR\cos t) \left( (R+d)^2u^2+  (R-d)^2  v^2-
	(R^2-d^2)^2 \right)=0.
	\label{eqn:i3p-eqn}
	\end{equation}
\end{proof}

Clearly, the semiaxis lengths of \eqref{eqn:i3p-eqn} are $R\pm d$ which is the goal of this proof.

\section{$E_1$ Axis Invariance}
\label{app:e1-proof}
\begin{proposition}
The semiaxes of $E_1$ are $\eta_1=R+d$ and $\zeta_1=R-d$.
\end{proposition}

\begin{proof}
Using the parametrization of the triangle $P_1(t),P_2(t),P_3(t)$ given by equation \eqref{eq:p1p2p3t} and that $E_1$ pass through the vertices $P_i(t)$ (i=1,2,3) and    centered in $X_1=(2d,0)$ it is obtained:

\begin{align*}
E_1(x,y)=& \left(   ({R}^{2}-{d}^{2})^{2}-4\,d R
\cos t  (R\cos t-d)^{2}- {R}^{2} \sin^2 t
   \right) {x}^{2}\\
+&\left( ({R}^{2}-{d}^{2})^{2}-8\,d{R}^{2} ( R\cos t -d ) \sin^2 t  \right) {y}^{2}\\
 -&4\,dR\sin t \left( 2\,R\cos t-R-d \right)  \left( 2\,R\cos t +R-d
 \right) xy\\
  +&4 d (4 dR \cos t((R\cos t-d)^2-R^2 \sin^2t)-(R^2-d^2)^2)x\\
  +&8 R d^2\sin t (2 R\cos t+R-d) (2 R\cos t-R-d)y\\
    -& 2 dR  \cos t (16 d^2R \cos t  (R  \cos t-d)-(R^2-d^2) (R^2+7 d^2))\\
    -&(R^2-3 d^2) (R^2-d^2) =0
\end{align*}
Proceeding as in the proof of Theorem~\ref{thm:i3p} it is direct to verify that the canonical form of the above equation is $u^2/(R+d)^2+v^2/(R-d)^2=1.$

\end{proof}

\end{document}